\documentclass[a4,12pt]{amsart}
\usepackage{amssymb,amstext,amsmath,amscd,amsthm,amsfonts,enumerate,latexsym}
\usepackage{color}
\usepackage[all]{xy}

\theoremstyle{plain}
\newtheorem{thm}{Theorem}[section]

\newtheorem*{thm*}{Theorem}
\newtheorem*{cor*}{Corollary}

\newtheorem{prop}[thm]{Proposition}

\newtheorem{lem}[thm]{Lemma}
\newtheorem{cor}[thm]{Corollary}

\newtheorem*{claim*}{Claim}
\newtheorem{defn}[thm]{Definition}

\theoremstyle{definition}
\newtheorem{ex}[thm]{Example}
\newtheorem{Example}[thm]{Example}

\theoremstyle{remark}
\newtheorem{rem}[thm]{Remark}

\numberwithin{equation}{thm}

\def\Min{\operatorname{Min}}

\def\Ext{\operatorname{Ext}}

\def\Ker{\operatorname{Ker}}

\def\Hom{\operatorname{Hom}}

\def\Max{\operatorname{Max}}

\def\m{\mathfrak m}

\def\p{\mathfrak p}
\def\q{\mathfrak q}

\def\K{\mathrm{K}}

\newcommand{\im}{\mathop{\mathrm{Im}}\nolimits}

\newcommand{\rmK}{\mathrm{K}}

\newcommand{\rmQ}{\mathrm{Q}}

\newcommand{\rmU}{\mathrm{U}}
\newcommand{\rmV}{\mathrm{V}}
\newcommand{\rmW}{\mathrm{W}}

\newcommand{\fka}{\mathfrak{a}}

\newcommand{\fkm}{\mathfrak{m}}

\newcommand{\fkp}{\mathfrak{p}}
\newcommand{\fkq}{\mathfrak{q}}

\newcommand{\mapright}[1]{%
\smash{\mathop{%
\hbox to 1cm{\rightarrowfill}}\limits^{#1}}}

\newcommand{\mapleft}[1]{%
\smash{\mathop{%
\hbox to 1cm{\leftarrowfill}}\limits_{#1}}}

\def\depth{\operatorname{depth}}

\def\Supp{\operatorname{Supp}}

\def\Ass{\operatorname{Ass}}

\def\height{\mathrm{ht}}
\def\Spec{\operatorname{Spec}}

\title[On the weakly Arf $(\mathrm{S}_2)$-ifications of Noetherian rings]{On the weakly Arf {\bfseries (}$\mathbf{S_2}${\bfseries )}-ifications of Noetherian rings}

\author[Naoki Endo]{Naoki Endo}
\address{Department of Mathematics, Faculty of Science Division II, Tokyo University of Science, 1-3 Kagurazaka, Shinjuku, Tokyo 162-8601, Japan}
\email{nendo@rs.tus.ac.jp}
\urladdr{https://www.rs.tus.ac.jp/nendo/}

\author[Shiro Goto]{Shiro Goto}
\address{Department of Mathematics, School of Science and Technology, Meiji University, 1-1-1 Higashi-mita, Tama-ku, Kawasaki 214-8571, Japan}
\email{shirogoto@gmail.com}

\author[Shin-ichiro Iai]{Shin-ichiro Iai}
\address{Mathematics laboratory, Sapporo College, Hokkaido University of Education, 1-3 Ainosato 5-3, Kita-ku, Sapporo 002-8502, Japan}
\email{iai.shinichiro@s.hokkyodai.ac.jp}

\author[Naoyuki Matsuoka]{Naoyuki Matsuoka}
\address{Department of Mathematics, School of Science and Technology, Meiji University, 1-1-1 Higashi-mita, Tama-ku, Kawasaki 214-8571, Japan}
\email{naomatsu@meiji.ac.jp}

\thanks{2020 {\em Mathematics Subject Classification.} 13A15, 13H15, 13H10.}
\thanks{{\em Key words and phrases.} weakly Arf ring, $(S_2)$-ification, global canonical module}
\thanks{The first author was partially supported by JSPS Grant-in-Aid for Young Scientists 20K14299.} 
\thanks{The second author was partially supported by JSPS Grant-in-Aid for Scientific Research (C) 21K03211. }

\begin{document}

\maketitle



\begin{abstract}
The weakly Arf $(S_2)$-ification of a commutative Noetherian ring $R$ is considered to be a birational extension which is good next to the normalization. The weakly Arf property ({\em WAP} for short) of $R$ was introduced in 1971 by J. Lipman with his famous paper \cite{L}, and recently rediscovered by \cite{CCCEGIM}, being closely explored with further developments. The present paper aims at constructing, for a given Noetherian ring $R$ which satisfies certain mild conditions, the smallest module-finite birational extension of $R$ which satisfies WAP and the condition $(S_2)$ of Serre. We shall call this extension the weakly Arf $(S_2)$-ification, and develop the basic theory, including some existence theorems.
\end{abstract}


\section{Introduction}
The purpose of this paper is to construct, for a given commutative Noetherian ring $R$,  the smallest module-finite birational extension of $R$, satisfying the weakly Arf property and the condition $(S_2)$ of Serre.

In 1971 J. Lipman \cite{L} introduced the notion of Arf ring and developed the basic theory, extending the results of C. Arf \cite{Arf}, concerning  the multiplicity sequences of curve singularities, to those on one-dimensional Cohen-Macaulay rings. Let $R$ be a Noetherian semi-local ring and assume that $R$ is a Cohen-Macaulay ring of dimension one, i.e., for every $M \in \Max R$ the local ring $R_M$ is Cohen-Macaulay and of dimension one. Then we say that $R$ is {\it an Arf ring}, if the following conditions are satisfied (\cite[Definition 2.1]{L}), where $\overline{R}$ denotes the integral closure of $R$ in its total ring $\rmQ(R)$ of fractions.
\begin{enumerate}[$(1)$]
\item Every integrally closed ideal $I$ in $R$ which contains a non-zerodivisor has {\it a principal reduction}, i.e., $I^{n+1} = a I^n$ for some $n \ge 0$ and $a \in I$.
\item Let $x,y,z \in R$ such that $x$ is a non-zerodivisor on $R$. If $\frac{y}{x}, \frac{z}{x} \in \overline{R}$, then $\frac{yz}{x} \in R$.
\end{enumerate}

\noindent
Here, we notice that, provided $\overline{R}$ is finite as an $R$-module, among the Arf rings between $R$ and  $\overline{R}$, there exists the smallest one, called {\it the Arf closure} of $R$.

In \cite{L} Lipman introduced also the notion of strict closure for arbitrary extensions of two commutative rings and developed the underlying theory in connection with a conjecture of O. Zariski. As is stated in \cite{L}, Zariski conjectured that the Arf closure of $R$ should coincide with the strict closure of $R$ in $\overline{R}$, provided $R$ is a one-dimensional Cohen-Macaulay ring, and Zariski himself proved that the Arf closure is contained in the strict closure  (\cite[Proposition 4.5]{L}). Lipman proved the conjecture in the case where $R$ contains a field (\cite[Corollary 4.8]{L}), and in \cite[Theorem 4.4]{CCCEGIM} the authors settled Zariski's conjecture with full generality, by looking into a slight modification of Arf closures which they call {\it the weak  Arf closures}.

The research \cite{CCCEGIM} was originally inspired by \cite{L} and aimed at a higher-dimensional generalization of the theory of Lipman. In \cite{CCCEGIM} the authors introduced the notion of weakly Arf ring, and developed the fundamental theory. 
\begin{defn}[Lemma \ref{defArf}]\label{1} Let $R$ be an arbitrary commutative ring and let $$\rmW(R) = \{a \in R \mid a~\text{is~a~non-zerodivisor~of}~R\}.$$ Then the following three conditions are equivalent.
\begin{enumerate}
	\item[$(1)$] The ring $R$ satisfies Condition $(2)$ in the above definition of Arf rings. 
	\item[$(2)$] $\overline{aR\, }^{\, 2}=a\,\overline{(aR)}$ for all $a\in\rmW(R)$, where $\overline{I}$ denotes, for an ideal $I$ of $R$, the integral closure of $I$.
	\item[$(3)$] $\dfrac{R}{a}\cap \overline{R}$ is a subring of $\overline{R}$ for every $a\in\rmW(R)$, where $\dfrac{R}{a}=a^{-1}R$ in $\rmQ(R)$.
\end{enumerate}
When this is the case, we say that  $R$ is a weakly Arf ring.
\end{defn}

\noindent
Let us  call Condition $(2)$ in the definition of Arf rings {\it the weakly Arf property} ({\it WAP} for short) of $R$. Both of strict and weakly Arf closures are defined for arbitrary commutative rings. This fact leads us to the natural question of when the two closures coincide, and it is proved by \cite[Corollary 7.7]{CCCEGIM} that with certain mild assumptions on $R$ (e.g., $R$ contains an infinite field), the answer to the question is affirmative, once the weakly Arf closure satisfies $(S_2)$. Nevertheless, as is pointed out by \cite[Example 4.3]{EGI}, in general the weakly Arf closure does not satisfy $(S_2)$, even if the strict and weakly Arf closures coincide.

This fact strongly urged us to start the present research. The main purpose is to show, for a given Noetherian ring $R$, the existence of the weakly Arf $(S_2)$-ification, i.e., the smallest module-finite birational extension of $R$ which satisfies both of WAP and $(S_2)$.

Before stating the result of this paper, let us give a little more comment about weakly Arf rings. As is stated in Definition \ref{1},  a commutative ring $R$ is weakly Arf if and only if $\overline{aR}^2=a\overline{(aR)}$ for every $a \in \rmW(R)$ (see \cite[Theorem 2.4]{CCCEGIM} also), while the integral closedness of $R$ is equivalent to saying that $\overline{aR}=aR$ for every $a \in \rmW(R)$. Namely, the WAP is originally  very close to the integral closedness, which is the reason, for a Noetherian ring $R$, why the weakly Arf $(S_2)$-ification is of interest next to the normalization $\overline{R}$, and may have its own significance.

With this observation the conclusion of this paper can be stated as follows.

\begin{thm}[Corollary \ref{integralcor3}]\label{main}
Let $R$ be a locally quasi-unmixed Noetherian ring which satisfies the condition $(S_1)$ of Serre. If $\overline{R}$ is a finitely generated $R$-module, then $R$ admits the weakly Arf $(S_2)$-ification.
\end{thm}

We are now in a position to explain how this paper is organized. In Section $2$ we explain the basic methods of constructing $(S_2)$-ifications. For a Noetherian ring $R$ which satisfies $(S_1)$, the ring 
$$
\widetilde{R}=\{f\in\rmQ(R)\mid I f\subseteq  R\ \text{for some ideal}\ I\ \text{with}\ \height_RI\ge 2\}
$$ 
is the $(S_2)$-ification of $R$ and has the form $\widetilde{R} = 
\dfrac{R}{a}\cap\dfrac{R}{b}$ in $\rmQ(R)$ for some $a,b \in \rmW(R)$, once 
$\widetilde{R}$ is a finitely generated  $R$-module (cf. \cite{AG,B,EGA4-2,GS,G,HH,Sc}). This naturally leads us to the problem of when the $(S_2)$-ifications $\widetilde{R}$ of $R$ is a Noetherian ring, or more specifically, when $\widetilde{R}$ is a finitely generated $R$-module, which we shall answer in Section 3. Section 2 is devoted to some preliminaries which we will later use in Section 4. We shall summarize also the basic results on the global canonical modules in Section 3, and applying these results, we eventually prove Theorem \ref{main} in the final section 4.

In what follows, unless otherwise specified, let $R$ denote an arbitrary commutative ring, and $\overline{R}$ the integral closure of $R$ in its total ring $\rmQ(R)$ of fractions. For an $R$-module $M$, let $\mathrm{W}(M)=\{ a \in R \mid a~\text{is~a~non-zerodivisor~for}~M\}$.

\section{Construction of $(S_2)$-ifications}

In this section, let $M$ be an $R$-module and set $W=\mathrm{W}(M)$.
We consider $M$ as an $R$-submodule of $W^{-1}M$ the localization of $M$ with respect to $W$. For each $a \in W$ and an $R$-submodule $L$ of $W^{-1}M$, we set 
$$
\frac{L}{a} =\left\{ \frac{x}{a} ~\bigg|~  x\in L \right\}  \ \subseteq \  W^{-1}M
$$  
which forms an $R$-submodule of $W^{-1}M$ containing $L$, where $\displaystyle \frac{x}{a} = a^{-1} x$ for all $x\in L$. 
We shall explore, for each $a, b \in W$, the $R$-submodule of $W^{-1}M$ of the form
$$
\frac{M}{a}\cap\frac{M}{b}=\frac{aM:_Mb}{a}.
$$

Recall that, for $a_1, a_2, \dots a_\ell \in R~(\ell>0)$ and an $R$-module $N$, the sequence $a_1, a_2, \dots a_\ell$ is called {\it $N$-regular}, if $a_i$ is a non-zerodivisor on $N/(a_1, a_2, \dots a_{i-1})N$ for every $1\le i\le \ell$. 
Here, we do not necessarily assume $N/(a_1, a_2, \dots a_\ell)N \ne (0)$.

We begin with the following, which is essentially given by \cite{GS, G}.

\begin{thm}[cf. \cite{GS, G}]\label{ab}
Let $a,b\in W$. Then the following two conditions are equivalent.
\begin{enumerate}
\item[$(1)$] The sequence $a, b$ is $\displaystyle \left(\frac{M}{a}\cap\frac{M}{b}\right)$-regular.
\item[$(2)$] The equality $\displaystyle \frac{M}{a}\cap\frac{M}{b}=\frac{M}{a^2}\cap\frac{M}{b^2}$ holds.
\end{enumerate}
When this is the case, provided $M=R$, the $R$-module $\dfrac{R}{a}\cap\dfrac{R}{b}$ is a subring of $\rmQ(R)$.
\end{thm}
\begin{proof}
We first prove the last assertion. Suppose $M=R$. By $(2)$, we have the equalities
$$
\dfrac{R}{a}\cap\dfrac{R}{b} = \frac{aR:_Rb}{a}=\frac{a^2R:_Rb^2}{a^2}. 
$$
This yields $a^2R:_Rb^2=a(aR:_Rb)$. Since $I=aR:_Rb$ is an ideal of $R$ with $I^2=aI$, the $R$-submodule $\dfrac{I}{a} = R\left[\dfrac{I}{a}\right]$ has a ring structure. 
Hence $\dfrac{R}{a}\cap\dfrac{R}{b}$ is a subring of $\rmQ(R)$.

$(2)\Rightarrow (1)$ By setting $L=\dfrac{M}{a}\cap\dfrac{M}{b}$, we get
$$
L=\dfrac{M}{a}\cap\dfrac{M}{b}\subseteq \dfrac{L}{a}\cap\dfrac{L}{b}\subseteq \dfrac{M}{a^2}\cap\dfrac{M}{b^2}=L,
$$ 
whence $L=\dfrac{L}{a}\cap\dfrac{L}{b}$. This shows, because $a \in W(L)$ and $L= \dfrac{aL:_Lb}{a}$, the sequence $a, b$ is $L$-regular.

$(1)\Rightarrow (2)$ Note that  $L=\dfrac{M}{a}\cap\dfrac{M}{b}\subseteq \dfrac{M}{a^2}\cap\dfrac{M}{b^2}$. Conversely, we take $f\in \dfrac{M}{a^2}\cap\dfrac{M}{b^2}$ and write 
$$
f=\dfrac{x}{a^2}=\dfrac{y}{b^2} \ \  \ \text{in} \ \ \ W^{-1}M
$$
for some $x,y\in M$. Since $M\subseteq L$, we have 
$b(bx)=a(ay)\in aL$. This implies $x\in aL$, because the sequence $a,b$ is $L$-regular. Thus $\dfrac{x}{a}\in L$. Similarly, because the sequence $b,a$ is also $L$-regular, we get $\dfrac{y}{b}\in L$. 
The equality $\dfrac{x}{a^2}=\dfrac{y}{b^2}$ implies $b\cdot \dfrac{x}{a}=a\cdot \dfrac{y}{b}\in aL,$ so that $\dfrac{x}{a}\in aL:_Lb = aL$. Hence $f = \dfrac{x}{a^2}\in L$, as desired. 
\end{proof}

\begin{Example}
Let $R=k[x,y]$ the polynomial ring over a field $k$. We set $\m=(x,y)R$ and 
$M=\left\langle 
\begin{pmatrix}
x\\0
\end{pmatrix},
\begin{pmatrix}
y\\x
\end{pmatrix},
\begin{pmatrix}
0\\y
\end{pmatrix}
\right\rangle\subseteq \begin{array}{l}
R\vspace{-1.5mm}\\ \oplus\vspace{-1mm}\\ R
\end{array}~(=R^2)$.
Then
\begin{center}
$
\dfrac{M}{x}\cap\dfrac{M}{y}=\begin{array}{l}
\m\vspace{-1mm}\\ \oplus\vspace{-1mm}\\ \m
\end{array}
$
\quad and\quad\ \
$
\dfrac{M}{x^n}\cap\dfrac{M}{y^n}=\begin{array}{l}
R\vspace{-1.5mm}\\ \oplus\vspace{-1mm}\\ R
\end{array}
$
\end{center}
for all integers $n\ge 2$.
\end{Example}

\begin{proof}
For each $n >0$, note that $(x^n R^2)\cap (y^n R^2)=x^ny^n R^2$ because $x^n, y^n$ is $R^2$-regular. This shows $x^nM\cap y^n M\subseteq x^ny^n R^2$, and the equality holds for $n\ge2$. Indeed, the equalities
$$\begin{pmatrix}
x^ny^n\\0
\end{pmatrix}=x^n\left(
y^{n-1}\begin{pmatrix}
y\\x
\end{pmatrix}+
xy^{n-2}\begin{pmatrix}
0\\y
\end{pmatrix}\right)
=y^n\left(
x^{n-1}\begin{pmatrix}
x\\0
\end{pmatrix}\right)
$$
imply  $\begin{pmatrix}
x^ny^n\\0
\end{pmatrix}\in x^nM\cap y^n M$. 
Similarly $\begin{pmatrix}
0\\x^ny^n
\end{pmatrix}\in x^nM\cap y^n M$.  Hence $x^nM\cap y^n M= x^ny^n R^2$, and therefore $\dfrac{M}{x^n}\cap\dfrac{M}{y^n}=\dfrac{x^nM\cap y^n M}{x^ny^n}=R^2$. To prove $\dfrac{M}{x}\cap\dfrac{M}{y}=\begin{array}{l}
\m\vspace{-1mm}\\ \oplus\vspace{-1mm}\\ \m
\end{array}
$,
it suffices to show that
$$
xM\cap yM=\left\langle 
\begin{pmatrix}
x^2y\\0
\end{pmatrix},
\begin{pmatrix}
xy^2\\0
\end{pmatrix},
\begin{pmatrix}
0\\x^2y
\end{pmatrix},
\begin{pmatrix}
0\\xy^2
\end{pmatrix}
\right\rangle.
$$
To do this, let $f\in xM\cap y M$ and write 
$$f=
x\left(f_1\begin{pmatrix}
x\\0
\end{pmatrix}+
f_2\begin{pmatrix}
y\\x
\end{pmatrix}+
f_3\begin{pmatrix}
0\\y
\end{pmatrix}\right)
=
y\left(g_1\begin{pmatrix}
x\\0
\end{pmatrix}+
g_2\begin{pmatrix}
y\\x
\end{pmatrix}+
g_3\begin{pmatrix}
0\\y
\end{pmatrix}
\right)
$$
where $f_i,g_j\in R$. Then $x(f_1x+f_2y)=y(g_1x+g_2y)$ and $x(f_2x+f_3y)=y(g_2x+g_3y)$. As $x, y$ is $R$-regular, we can choose $h_1, h_2, h_3, h_4\in R$ such that $f_1x+f_2y=yh_1$, $f_2x+f_3y=yh_2$, $g_1x+g_2y=xh_3$, $g_2x+g_3y=xh_4$. Therefore  $f_1=yh_5$, $f_2=yh_6$, $g_2=xh_7$, $g_3=xh_8$ for some $h_5, h_6, h_7, h_8\in R$.
Then, because $f_3=-xh_6+xh_7+yh_8$, we get
\begin{align*}
f&=xf_1\begin{pmatrix}
x\\0
\end{pmatrix}+
xf_2\begin{pmatrix}
y\\x
\end{pmatrix}+
xf_3\begin{pmatrix}
0\\y
\end{pmatrix}\\
&=x(yh_5)\begin{pmatrix}
x\\0
\end{pmatrix}+
x(yh_6)\begin{pmatrix}
y\\x
\end{pmatrix}+
x(-xh_6+xh_7+yh_8)\begin{pmatrix}
0\\y
\end{pmatrix}\\
&=
h_5\begin{pmatrix}
x^2y\\0
\end{pmatrix}+
h_6\begin{pmatrix}
xy^2\\0
\end{pmatrix}+
h_7\begin{pmatrix}
0\\x^2y
\end{pmatrix}+
h_8\begin{pmatrix}
0\\xy^2
\end{pmatrix}. 
\end{align*}
It is straightforward to check the converse. This completes the proof.
\end{proof}

In what follows, we assume $M$ is a torsion-free $R$-module and set
$$
V=\mathrm{Q}(R)\otimes_R M.
$$
Hence, every $f \in V$ has an expression of the form $f = \frac{m}{a}$ where  $a \in \rmW(R)$ and $m \in M$.

\medskip

Let $\operatorname{Ht}_{\ge 2}(R)$ denote the set of all ideals of $R$ which have height at least $2$. We consider $\height_RR=\infty$ for convention. Hence $R\in\operatorname{Ht}_{\ge 2}(R)$. Define
$$
\widetilde{M}=\{f\in V \mid If\subseteq M\ \mathrm{for}\ \mathrm{some}\ I\in\operatorname{Ht}_{\ge 2}(R)\} \subseteq V.
$$
Since $IJ \in \operatorname{Ht}_{\ge 2}(R)$ for all $I, J \in \operatorname{Ht}_{\ge 2}(R)$, $\widetilde{M}$ is an $R$-submodule of $V$ containing the $R$-module $M$.
For an $R$-submodule $N$ of $V$ with $M\subseteq N\subseteq V$, we then have $\widetilde{\,N\, }$ is the $R$-submodule of $V$, whence $\widetilde{M}\subseteq \widetilde{\,N\, }$. 
When $M=R$, we identify $V$ with $\rmQ(R)$, so the $R$-submodule $\widetilde{R}$ is a subring of $\rmQ(R)$. Then $\widetilde{M}$ is an $\widetilde{R}$-submodule of $\rmQ(R)$.

We now summarize some basic results on the $R$-module $\widetilde{M}$. A part of them will be also appeared in the forthcoming paper \cite[Section 2]{EGIM}. 
Because the results play an important role in our argument, let us include brief proofs for the sake of completeness.

\begin{prop}\label{basic}
Let $a,b\in \mathrm{W}(R)$. If $\height_R(a,b)\ge 2$, then the sequence $a,b$ is $\widetilde{M}$-regular.
\end{prop}

\begin{proof}
Let $f \in \widetilde{M}$. Assume that $bf=ag$ for some $g\in \widetilde{M}$. We choose $I,J\in\operatorname{Ht}_{\ge 2}(R)$ such that $If, Jg\subseteq M$. Set $\varphi=\frac{f}{a} = \frac{g}{b}$. Then, since $I(a\varphi)\subseteq M$ and $J(b\varphi)\subseteq M$, we get $(Ia+Jb)\varphi\subseteq M$. Hence $\varphi\in\widetilde{M}$ because $Ia+Jb\in \operatorname{Ht}_{\ge 2}(R)$. This shows $f\in a\widetilde{M}$.
\end{proof}

In the rest of this section, we furthermore assume $R$ is a Noetherian ring.

\begin{prop}\label{smallest}
Let $N$ be an $R$-submodule of $V$ such that $M\subseteq N\subseteq V$. Assume that, for all $a,b\in \mathrm{W}(R)$ with $\height_R(a,b)\ge 2$, the sequence $a,b$ is $N$-regular. Then $\widetilde{M}\subseteq N$.
\end{prop}

\begin{proof}
Let $f\in \widetilde{M}$ and write $f=\dfrac{x}{a}$ where $x\in M$ and $a\in \mathrm{W}(R)$. Consider the ideal $I=M:_Rf$. We then have $a\in I$ and $I\in \operatorname{Ht}_{\ge 2}(R)$. Since 
$$
I\not\subseteq \bigcup_{\p\in \Min_RR/aR}\p \ \cup \bigcup_{\p\in \Ass R}\p,
$$
we choose $b\in \mathrm{W}(R)\cap I$ satisfying $\height_R(a,b)\ge 2$. The sequence $a,b$ is $N$-regular. As $bf\in M\subseteq N$, we can write $bx=ay$ for some $y\in N$. Hence $x\in aN:_Nb = aN$. This implies $f\in N$.
\end{proof}

Hence we have the following. 

\begin{cor}\label{tt}
The following two conditions are equivalent:
\begin{enumerate}
\item[$(1)$] $M=\widetilde{M}$.
\item[$(2)$] For every $a,b\in \mathrm{W}(R)$ with $\height_R(a,b)\ge 2$, the sequence $a,b$ is $M$-regular.
\end{enumerate}
In particular, the equality $\widetilde{M}=\widetilde{\widetilde{M}}$ holds.
\end{cor}

Let $n > 0$ be an integer and $N$ a nonzero finitely generated $R$-module. We say that $N$ satisfies {\it the condition $(S_n)$ of Serre}, if 
$
\depth_{R_\p}N_\p\ge\min\{n, \dim R_\p\}
$
for every $\p\in \Supp_RN$.
The lemma below follows from the induction on $n$.

\begin{lem}\label{sequence}
Let $n > 0$ be an integer and $N$ a nonzero finitely generated $R$-module. 
Then the following conditions are equivalent.
\begin{enumerate}
\item[$(1)$] $N$ satisfies $(S_n)$.
\item[$(2)$] Every sequence $a_1, a_2, \dots , a_n\in R$ satisfying $\height_R(a_1, a_2, \dots , a_i)\ge i$ for all $1\le i\le n$ is $N$-regular. 
\end{enumerate}
If $R$ satisfies $(S_1)$, then one can add the following.
\begin{enumerate}
\item[$(3)$] Every sequence $a_1, a_2, \dots, a_n\in \mathrm{W}(R)$ satisfying $\height_R(a_1, a_2, \dots , a_i)\ge i$ for all $1\le i\le n$ is $N$-regular. 
\end{enumerate}
\end{lem}

To sum up this kind of arguments, we get the following.

\begin{thm}\label{S2}
Suppose that $R$ satisfies $(S_1)$ and $\widetilde{M}$ is finitely generated as an $R$-module. Then $\widetilde{M}$ is the smallest $R$-submodule of $V$ which contains $M$ and satisfies $(S_2)$.
\end{thm}

Notice that, even if $M$ is a finitely generated $R$-module, $\widetilde{M}$ is not necessarily finitely generated as an $R$-module. 
Theorem \ref{S2} leads to the problem of when the $R$-module $\widetilde{M}$ is finitely generated as an $R$-module, which we shall discuss in Section 3.

\begin{Example}
Let $T=k[[X,Y,Z]]$ the formal power series ring over a field $k$. Consider the rings $R=T/(X)\cap (Y,Z)$, $A=T/(X)$, and $B=T/(Y,Z)$. We identify $\rmQ(R)=\rmQ(A)\times\rmQ(B)$. Then $\widetilde{R}=A\times \rmQ(B)$. Hence $\widetilde{R}$ is not finitely generated as an $R$-module.
\end{Example}

\begin{proof}
Let $\varphi\in \widetilde{R}$ and write $\varphi=(\xi, \eta )$ for some $\xi \in \rmQ(A)$ and $\eta \in \rmQ(B)$. We then choose $I\in\operatorname{Ht}_{\ge 2}(R)$ such that $I\varphi\subseteq R$. As $R\subseteq \overline{R}=A\times B$, we obtain $I\xi\subseteq A$. We may assume $I\neq R$. Choose a system $a,b$ of parameters for $R$ with $(a,b)\subseteq I$. The sequence $a,b$ forms a system of parameters for $A$. Since $(a,b)\xi\subseteq A$ and the sequence $a,b$ is $A$-regular, we get $a\xi\in aA:_Ab = aA$. 
Hence $\xi\in A$ and $\widetilde{R}\subseteq  A\times \rmQ(B)$.

Conversely, let $\varphi\in A\times \rmQ(B)$ and write $\varphi=(\xi, \eta )$ with $\xi \in A$ and $\eta \in \rmQ(B)$. By setting $x$ the image of $X$ in $B$, we have $B= k[[x]]$, whence we can write $\eta =\frac{\beta}{x^{n}}$ for some $n\ge 0$ and $\beta\in B$. We consider $a=X^{n+1}+Y$, $b=X^{n+1}+Z$ and $I=(a,b)R$. Notice that  $\height_RI=2$, $IA=(Y,Z)A$, and $IB=x^{n+1}B$. This implies $I\xi=(Y,Z)\xi$ and $I\eta=x^{n+1}B\frac{\beta}{x^{n}}=X\beta B$. For every $i\in I$, we see that 
\begin{center}
$i\xi =\overline{YF_1+ZF_2}$ \  in  \ $A$ \ \ and \ \  $i\eta =\overline{XF_3}$ \ in \ $B$
\end{center}
 for some $F_1, F_2, F_3\in T$, where $\overline{(*)}$ denotes the image of $*$ in the appropriate rings. Then
$$
i\varphi=(i\xi, i\eta)=\Big(\overline{YF_1+ZF_2+XF_3},~\overline{YF_1+ZF_2+XF_3}\Big).
$$
This yields $I\varphi\subseteq R$, whence $\varphi\in\widetilde{R}$. Therefore $\widetilde{R}=A\times \rmQ(B)$, as desired. 
\end{proof}

\medskip

The smallest module-finite birational extension of $R$ satisfying the condition $(S_2)$ of Serre is called {\it the $(S_2)$-ification} of $R$. 
We apply Theorem \ref{S2} to get the following.

\begin{cor}\label{S2ring}
Suppose that $R$ satisfies $(S_1)$ and $\widetilde{R}$ is a finitely generated  $R$-module. Then $\widetilde{R}$ is the $(S_2)$-ification of $R$.
\end{cor}

Let $S$ be a Noetherian ring which is a birational extension of $R$.
Note that, provided $S$ is a finitely generated $R$-module, $S$ satisfies $(S_2)$ as a ring (i.e., $S$ satisfies $(S_2)$ as an $S$-module) if and only if it satisfies $(S_2)$ as an $R$-module. 
Although $\widetilde{R}$ is not necessarily finitely generated as an $R$-module, the ring $\widetilde{R}$ is integral over $R$, provided $R$ is locally quasi-unmixed (see \cite{M}). 

\begin{prop}\label{as a ring}
Suppose that $R$ satisfies $(S_1)$ and $\widetilde{R}$ is a Noetherian ring which is integral over $R$. Assume that $\widetilde{M}$ is a finitely generated $\widetilde{R}$-module. Then $\widetilde{M}$ satisfies $(S_2)$ as an $\widetilde{R}$-module. In particular, $\widetilde{R}$ satisfies $(S_2)$ as a ring.
\end{prop}

\begin{proof}
Let $P\in\Spec \widetilde{R}$ such that $\depth_{\widetilde{R}_P} \widetilde{M}_P<2$. 
We will show that $\widetilde{M}_P$ is a maximal Cohen-Macaulay $\widetilde{R}_P$-module. To do this, we may assume $\height_{\widetilde{R}}\, P\ge 1$. Set $\p=P\cap R$. 
Suppose $\height_{\widetilde{R}}\, P\ge 2$. Then $\height_{R}\,\p\ge 2$, so there exist $a,b\in \p\cap\mathrm{W}(R)$ such that $\height_R\,(a,b)\ge 2$. By Proposition \ref{basic}, the sequence $a,b$ is $\widetilde{M}$-regular. This contradicts $\depth_{\widetilde{R}_P}\widetilde{M}_P<2$. 
Therefore $\height_{\widetilde{R}}\,P=1$, so that $\height_{R}\,\p=1$. Hence we can choose $a\in\p\cap \rmW(R)$. Since $a$ is $\widetilde{M}$-regular, we get $\depth_{\widetilde{R}_P} \widetilde{M}_P=1$. Therefore $\widetilde{M}_P$ is a maximal Cohen-Macaulay $\widetilde{R}_P$-module. 
\end{proof}

In addition, we assume $M$ is finitely generated as an $R$-module. For $a \in \rmW(R)$, we set
$$
\Min^1_RM/aM=\{\p\in \Min_RM/aM\mid \height_R\p=1\}.
$$
Then $\Min^1_RM/aM=\Min_RM/aM=\Min_RR/aR$, whenever $(0):_RM=(0)$. 

\medskip

A Noetherian local ring $(A, \fkm)$ is said to be {\it quasi-unmixed}, if all the minimal prime ideals of the $\fkm$-adic completion $\widehat{A}$ of $A$ have the same codimension. A Noetherian ring $S$ is said to be {\it locally quasi-unmixed}, if $S_P$ is quasi-unmixed for every $P\in\Spec S$.

We note the following. 

\begin{lem}
Suppose that $R$ is locally quasi-unmixed and  satisfies $(S_1)$. Then $\dim_{R_\p}M_\p= \dim R_\p$ for all $\p \in \Supp_R M$. Consequently, $
\Min^1_RM/aM=\Min_R M/aM
$,
if $a \in \rmW(R)$.
\end{lem}

\begin{proof}
Let $\p \in \Supp_RM$ and set $\dim_{R_\p}M_\p=n$. Choose $\q \in \Min_RM$, so that $\q \subseteq \p$ and $\height_{R/\q} \p/\q=n$. We then have $\q \in \Min R$, because $M$ is torsionfree and $R$ satisfies $(S_1)$. Therefore, $\height_R \, \p = n$, because $R_\p$ is quasi-unmixed and $\q R_\p \in \Min R_\p$.  Thus $\dim_{R_\p}M_\p= \dim R_\p$.
\end{proof}

Let $a \in \rmW(R)$ and
$$
aM = \bigcap_{\fkp \in \Ass_RM/aM}Q(\fkp)
$$
be a primary decomposition of $aM$ in $M$, 
where $Q(\fkp)$ denotes the $\p$-primary component of $aM$ in $M$. 
We set
$$
\rmU(aM)
=\begin{cases}
M,&\text{if}\ \ \Min^1_RM/aM= \emptyset,\\
\bigcap_{\fkp \in \Min^1_RM/aM}Q(\fkp),& \text{if}\ \ \Min^1_RM/aM\neq \emptyset.
\end{cases}
$$

\begin{lem}\label{a:b}
For each $a\in\mathrm{W}(R)$ with $\Min^1_RM/aM\neq \emptyset$, there exists $b\in\mathrm{W}(R)$ such that $\height_R(a,b)\ge 2$ and $\mathrm{U}(aM)=aM:_Mb$.
\end{lem}

\begin{proof}
Let $a\in\mathrm{W}(R)$ such that $\Min^1_RM/aM\neq \emptyset$. 
We may assume $\Ass_R M/aM \neq \Min^1_RM/aM$. Notice that, for every $\p\in(\Ass_R M/aM) \setminus (\Min^1_RM/aM)$, $\height_R\p\ge 2$. Consider 
$$
\fka =\prod_{\fkp \in (\Ass_R M/aM) \setminus (\Min^1_RM/aM)}\fkp
$$ 
and choose an integer $\ell>0$ such that
$\mathrm{U}(aM)= aM:_M\fka^\ell$. We then have
$\height_R\fka\ge 2$ and $\fka\cap \rmW(R)\neq \emptyset$. 
Hence
$$
\fka^\ell \not\subseteq \bigcup_{\fkq \in \Min_R R/aR}\fkq \ \cup \bigcup_{\fkq \in \Ass R}\fkq. 
$$ 
Now we take an element $b \in \fka^\ell$ but 
$
b \not\in \bigcup_{\fkq \in \Min_R R/aR}\fkq \ \cup \ \bigcup_{\fkq \in \Ass R}\fkq. 
$
Then $b\in\rmW(R)$, $\height_R(a,b)\ge 2$, and $\mathrm{U}(aM)=aM:_Mb$, as desired.
\end{proof}

\begin{thm}\label{intersection}
Let $a\in \mathrm{W}(R)$. Then $\mathrm{U}(aM)=a\widetilde{M}\cap M$.
\end{thm}

\begin{proof}
Let $a\in \mathrm{W}(R)$. If $\Min^1_RM/aM=\emptyset$, then $\dfrac{M}{a} \subseteq \widetilde{M}$ because 
\begin{center}
$\height_R ([(0):_RM]+aR) \ge 2$ and $([(0):_RM]+aR){\cdot}\dfrac{M}{a} \subseteq M$. 
\end{center}
Hence we have $\mathrm{U}(aM)=M=a\widetilde{M}\cap M$ in this case.
Thus we may assume  $\Min^1_RM/aM\neq\emptyset$. 
Let $m\in a\widetilde{M}\cap M$ and write $m=af$ for some $f\in \widetilde{M}$.
Then there exists an ideal $I \in \operatorname{Ht}_{\ge 2}(R)$ so that $If\subseteq M$. Hence $\height_R (aM:_Rm) \ge 2$ by $I \subseteq M:_Rf = aM:_Rm$.
Therefore we have $m \in aM_\p \cap M = Q(\p)$ for every $\p \in \Min^1_RM/aM$ which shows that $m \in \bigcap_{\fkp \in \Min^1_RM/aM}Q(\p) = \mathrm{U}(aM)$. 
Next, we prove the converse inclusion. 
By Lemma \ref{a:b}, there exists $b\in\mathrm{W}(R)$ such that $\height_R(a,b)\ge 2$ and $\mathrm{U}(aM)=aM:_M b$. 
We then have $(a,b)\cdot\dfrac{\rmU(aM)}{a}\subseteq M$ which implies $\dfrac{\rmU(aM)}{a}\subseteq\widetilde{M}$. Therefore $\rmU(aM)= a\widetilde{M}\cap M$, as claimed.
\end{proof}

We summarize some consequences. 

\begin{cor}
The equality $\widetilde{M}= \displaystyle\bigcup_{a \in \rmW(R)}\dfrac{\rmU(aM)}{a}$ holds.
\end{cor}

\begin{proof}
By Theorem \ref{intersection}, we have $\rmU(aM)\subseteq a\widetilde{M}$ for all $a \in \rmW(R)$. Conversely, let $f\in\widetilde{M}$ and write $f=\dfrac{x}{a}$ where $x\in M$ and $a \in \rmW(R)$. Then $x=af\in a\widetilde{M}\cap M=\mathrm{U}(aM),$ and hence $f\in \dfrac{\rmU(aM)}{a}$. 
\end{proof}

\begin{cor}[{\cite[Corollary 10.5]{M}}]\label{integral}
If $R$ is locally quasi-unmixed, then $\widetilde{R}\subseteq \overline{R}$.
\end{cor}

\begin{proof}
Since $\overline{aR}=a\overline{R}\cap R$ for all $a \in \rmW(R)$, we have 
$$
\overline{R}=\bigcup_{a \in \rmW(R)}\frac{\overline{aR}}{a}.
$$
It suffices to show $\rmU(aR)\subseteq \overline{aR}$ for all $a \in \rmW(R)$. Suppose the contrary, i.e., $\rmU(aR)\not\subseteq \overline{aR}$ for some $a \in \rmW(R)$, and seek a contradiction. Indeed, we consider the $R$-module
$$
N=(\rmU(aR)+\overline{aR})/\overline{aR}
$$ 
and take $\p\in\Ass_RN$. Then $\p\in\Ass_RR/\overline{aR}$. Since $R$ is locally quasi-unmixed, by \cite[Theorem 2.12]{R}, we see that $\height_R\p=1$. So $\rmU(aR)_\p=aR_\p$, and hence $N_\p=(0)$. This makes a contradiction. Therefore $\rmU(aR)\subseteq \overline{aR}$ for all $a \in \rmW(R)$.
\end{proof}

\begin{cor}\label{cor3}
The following three conditions are equivalent.
\begin{enumerate}
	\item[$(1)$] $\widetilde{M}$ is a finitely generated $R$-module.
	\item[$(2)$] $\mathrm{U}(aM)=a\widetilde{M}$ for some $a\in\mathrm{W}(R)$.
	\item[$(3)$] $(M:_R\widetilde{M})\cap \mathrm{W}(R)\neq \emptyset$.
\end{enumerate}
When this is the case, one has $\widetilde{M}=\dfrac{\mathrm{U}(aM)}{a}$ for some $a\in\mathrm{W}(R)$.
\end{cor}

We now reach the goal of this section.

\begin{thm}\label{conclusionab}
Suppose that $\widetilde{M}$ is a finitely generated $R$-module. Then there exist $a,b\in\rmW(R)$ such that 
$$
\widetilde{M}=\dfrac{M}{a}\cap\dfrac{M}{b}.
$$
Hence, if $R$ satisfies $(S_1)$, then $\dfrac{M}{a}\cap\dfrac{M}{b}$ is the smallest $R$-submodule of $V$ which contains $M$ and satisfies $(S_2)$. 
\end{thm}

\begin{proof}
By Corollary \ref{cor3}, we can choose $a\in\rmW(R)$ such that $\widetilde{M}=\dfrac{\mathrm{U}(aM)}{a}$. Thanks to Lemma \ref{a:b}, there exists $b\in\mathrm{W}(R)$ such that $\mathrm{U}(aM)=aM:_Mb$. Hence $
\widetilde{M}=\dfrac{M}{a}\cap\dfrac{M}{b}$, and the last assertion follows from Theorem \ref{S2}.
\end{proof}

\section{Global canonical modules and module-finiteness of $\widetilde{M}$}

This section aims at considering the problem of when the $R$-module $\widetilde{M}$ is finitely generated. To attack this problem, we apply the theory of global canonical module given by R. Sharp \cite{S}. 

Let $R$ be a Noetherian ring and assume that $R$ is a homomorphic image of a Gorenstein ring. We set 
$$
R=S/I,
$$
where $I$ is an ideal of a Gorenstein ring $S$. Let 
$$
I = \bigcap_{P \in \Ass_S R}Q(P)
$$
be a primary decomposition of $I$ in $S$, where $Q(P)$ stands for a $P$-primary component. For each ideal $J$ of $S$, we set $\rmV(J)=\{P\in\Spec S\mid J\subseteq P\}$. 
Then 
$$
\rmV(I)=\bigsqcup_{i=1}^nC_i
$$
where $C_1, C_2,\dots C_n$ are the connected components of $\rmV(I)$, i.e., the equivalence classes of the equivalence relation $\sim$ on $\rmV(I)$, which is given by, for $P,P'\in\rmV(I)$, $P\sim P'$ if and only if there exist an integer $\ell>0$ and a sequence $P_1, P_2, \dots , P_{\ell+1}$ in $\rmV(I)$ such that $P_1=P$, $P_{\ell+1}=P'$, and $P_{i}+P_{i+1}\neq S$ for all $1\le i\le \ell$. We put
$$
I_i=\bigcap_{P\in (\Ass_S R) \cap  C_i}Q(P).
$$
We then have $\bigcap_{i=1}^nI_i=I$ and $C_i=\mathrm{V}(I_i)$ for all $1\le i\le n$. Since $I_i+I_j=S$ for all $1\le i,j\le n$ with $i\neq j$, we have an isomorphism
$$
R\cong S/I_1\times \dots \times S/I_n
$$
of rings. 
We define
$$
\rmK_R=\bigoplus_{i=1}^n\Ext_S^{t_i}(S/I_i, S)
$$
and call it {\it the global canonical module} of $R$, where $t_i=\height_SI_i$ for $1\le i\le n$.
Notice that $\K_R$ is a finitely generated $R$-module, and $\K_R$ coincides with the ordinary canonical module provided $R$ is a local ring.

\begin{Example}\label{ex3}
Let $S=k[X,Y]$ the polynomial ring over a field $k$. We consider 
$$
Q_1=(X),\ Q_2=(1+Y),\ Q_3=(1+XY),\ Q_4=(X^2,Y),\ Q_5=(1+X+XY,~ Y^2),$$ and $I=\bigcap_{i=1}^5Q_i$.
The ring $R=S/I$ is not Cohen-Macaulay and $R\cong  S/I_1\times S/I_2$ as a ring, where $I_1=\bigcap_{i=1}^4Q_i$ and $I_2=Q_5$.
By setting $U_1=\bigcap_{i=1}^3Q_i$ which the principal ideal of $S$, we have the isomorphisms 
$$
\Ext_S^{1}(S/I_1, S)\cong \Ext_S^{1}(S/U_1, S)\cong S/U_1
$$ 
of $R$-modules. Since $I_2$ is generated by the $S$-regular sequence, we get
$$
\Ext_S^{2}(S/I_2, S)\cong S/I_2
$$ 
as an $R$-module. Therefore $\rmK_R\cong S/U_1\oplus S/I_2$ as an $R$-module.
\end{Example}

We now summarize some basic results of $\K_R$. 

\begin{prop}\label{AssKR}
$\Ass_R\K_R=\bigcup_{i=1}^n\{PR\mid P\in\mathrm{V}(I_i), \height_{S}P=\height_{S}I_i\}$.
\end{prop}

\begin{proof}
Let 
$
0\to S\to E^0\to E^1\to \dots \to E^{j}\to \cdots
$ 
be the minimal injective resolution of $S$. 
We fix an integer $i$ with $1\le i\le n$. 
Since $\Hom_S(S/I_i, E^j)=(0)$ for all integers $j<t_i$, 
we get
the monomorphism $\Ext_S^{t_i}(S/I_i, S)\hookrightarrow\Hom_S(S/I_i, E^{t_i})$. Hence 
$$
\Ass_S\Ext_S^{t_i}(S/I_i, S)\subseteq \mathrm{V}(I_i)\cap\Ass_SE^{t_i}=\{P\mid P\in\mathrm{V}(I_i), \height_{S}P=t_i\}.
$$
This shows $\Ass_R\K_R\subseteq \bigcup_{i=1}^n\{PR\mid P\in\mathrm{V}(I_i), \height_{S}P=t_i\}$.
Conversely, let $P\in\mathrm{V}(I_i)$ such that $ \height_{S}P=t_i$. Since 
$$
\dim S_P-\dim (S/I_i)_P=\height_{S_P}(I_i)_P=t_i,
$$
we then have $P\in\Supp_S\Ext_S^{t_i}(S/I_i, S)$, while $P\in\Min_S\Ext_S^{t_i}(S/I_i, S)$. Therefore $P\in\Ass_S\Ext_S^{t_i}(S/I_i, S)\subseteq \Ass_R\K_R$, as desired. 
\end{proof}

\begin{lem}\label{localizationKR}
$(\K_R)_\p\cong \K_{R_\p}$ as an $R_\p$-module for every $\p\in\Supp_R\K_R$.
\end{lem}

\begin{proof}
Let $\p\in\Supp_R\K_R$. Then there exist an integer $i$ and $P\in\mathrm{V}(I_i)$ such that $1\le i\le n$ and $\p=PR$. Choose $\q\in\Ass_R\K_R$ such that $\q\subseteq \p$. We write $\q=QR$ for some $Q\in\mathrm{V}(I_i)$. Then $\height_{S}Q=t_i$. Since $Q\subseteq P$, we have $\height_{S_P}(I_i)_P=t_i$. Hence 
$$
\dim S_P-\dim (S/I_i)_\p=\dim S_P-\dim S_P/(I_i)_P=\height_{S_P}(I_i)_P=t_i
$$
so we obtain the isomorphisms 
$$
(\K_R)_\p\cong \Ext_S^{t_i}(S/I_i, S)_\p\cong \Ext_{S_P}^{t_i}((S/I_i)_\p, S_P)\cong\K_{((S/I_i)_\p)}\cong\K_{(R_\p)}
$$
of $R_\p$-modules. 
\end{proof}

Let us now state the first main theorem of this section.

\begin{thm}\label{q-unmixed}
Suppose that $R$ is locally quasi-unmixed. Then $$\Ass_R\K_R=\Min R.$$ Hence $(\K_R)_\p\cong \K_{R_\p}$ as an $R_\p$-module for every $\p\in\Spec R$.
\end{thm}

\begin{proof}
Let $i$ be an integer with $1\le i\le n$. It suffice, by Proposition \ref{AssKR}, to show that
$$
\Min_S S/I_i\subseteq \{P\in\mathrm{V}(I_i)\mid \height_{S}P=\height_SI_i\}.
$$
Take $Q\in\Min_SS/I_i$. Let $Q'\in\Min_SS/I_i$ such that $\height_SQ'=\height_SI_i$. We shall show $\height_{S}Q=\height_SI_i$. Since $Q, Q'\in\mathrm{V}(I_i)$, there exist an integer $\ell >0$ and a sequence $Q_1, Q_2, \dots , Q_{\ell+1}\in\mathrm{V}(I_i)$ such that $Q_1=Q$, $Q_{\ell+1}=Q'$, and $Q_j+Q_{j+1}\neq S$ for all $1\le j\le \ell$. We may assume $Q_1, Q_2, \dots , Q_{\ell+1}\in\Min_S S/I_i$. We fix an integer $j$ with $1\le j\le \ell$. 
There exists $P\in\Max S$ such that $Q_j+Q_{j+1}\subseteq P$, and then $P\in\mathrm{V}(I_i)$. Since $R_P\cong (S/I_i)_P$ and $R$ is locally quasi-unmixed, we have $\height_{S/Q_j}P/Q_j=\height_{S/Q_{j+1}}P/Q_{j+1}$, so that $\height_{S}Q_j=\height_{S}Q_{j+1}$ because $S$ is a Cohen-Macaulay ring. Consequently, $\height_{S}Q=\height_{S}Q'$, as wanted. The last assertion follows from Lemma \ref{localizationKR}.
\end{proof}

\begin{lem}\label{endKR}
The endomorphism ring $\Hom_R(\K_R,\K_R)$ is commutative.
\end{lem}

\begin{proof}
Set $T=R\setminus \bigcup_{\q\in\Ass_R\K_R}\q$. The ring $\Hom_R(\K_R,\K_R)$ is naturally isomorphic to a subring of the commutative ring $T^{-1}R$.
\end{proof}

Let $\varphi:R\to \Hom_R(\K_R,\K_R)$ be the natural homomorphism of rings. We denote by 
$
(0)=\bigcap_{\q\in\Ass R}Q(\q)
$
the primary decomposition of $(0)$ in $R$, and set 
$$
U=\bigcap_{\q\in\Ass_R \K_R}Q(\q).
$$

With this notation we have the following.

\begin{lem}\label{U}
The equality $U=(0):_R\K_R$ holds. Hence, one has $U=\Ker\varphi$.
\end{lem}

\begin{proof}
Suppose $U\K_R\neq (0)$ and take $\p\in\Ass_R U\K_R$. Since $\p\in\Ass_R \K_R\subseteq \Min R$, we get $U_\p=(0)$ and hence $(U\K_R)_\p=(0)$. This is absurd. Therefore $U\K_R= (0)$, so that $U\subseteq (0):_R\K_R$.
Suppose $U\subsetneq (0):_R\K_R$ and seek a contradiction. Let $\p\in\Ass_R [(0):_R\K_R]/U$. Since $\p\in\Ass_R R/U$, we get $\p\in\Ass_R\K_R$, so that
$(\K_R)_\p\cong\K_{(R_\p)}$. Since $\dim R_\p=0$, we have $(0):_{R_\p}\K_{(R_\p)}=(0)$, whence $[(0):_R\K_R]_\p=(0)$. This makes a contradiction. Therefore $(0):_R\K_R=U$, as wanted. 
\end{proof}

Hence, if $\Ass R=\Ass_R\K_R$, then the canonical map $\varphi:R\to \Hom_R(\K_R,\K_R)$ is injective and $\mathrm{Q}(R)=T^{-1}R$, where $T=R\setminus \bigcup_{\q\in\Ass_R\K_R}\q$. Therefore, $\Hom_R(\K_R,\K_R)$ is naturally isomorphic to a birational extension of $R$.

\begin{Example}
We maintain the same notation as in Example \ref{ex3}. Since $\rmK_R\cong S/U_1\oplus S/I_2$ as an $R$-module, we have 
$$
\Hom_R(\K_R,\K_R)\cong S/U_1\times S/I_2
$$ 
as a ring, which is a Gorenstein ring. The ideal $U=(0):_R\K_R$ coincides with the nonzero ideal $(U_1\cap I_2)R$ of $R$, so the canonical map $\varphi$ is not injective.
\end{Example}

From now on, let us consider the finite generation of the $R$-module $\widetilde{M}$. We begin with the following.

\begin{lem}\label{abKR}
Let $a,b\in \mathrm{W}(R)$. If $\height_R(a,b)\ge 2$, then the sequence $a,b$ is $\K_R$-regular.
\end{lem}

\begin{proof}
As $\Ass_R \K_R\subseteq \Ass_R R$, we see that $a \in \rmW(\rmK_R)$. Suppose $b$ is a zerodivisor on $\K_R/a\K_R$. Then $\K_R/a\K_R\neq (0)$ and there exists $\p\in\Ass_R \K_R/a\K_R$ such that $b\in\p$. Since $\depth_{R_\p}(\K_R/a\K_R)_\p=0$, we have $\depth_{R_\p}\K_{R_\p}=1$. Therefore $\dim R_\p=1$, because $\K_{R_\p}$ is the ordinary canonical module of $R_\p$. This contradicts that $\height_R\,\p\ge 2$. Hence the sequence $a,b$ is $\K_R$-regular.
\end{proof}

Let $(-)^\vee=\Hom_R(-,\K_R)$ denote the $\K_R$-dual. For a finitely generated $R$-module $M$, by Proposition \ref{AssKR},  we have
$$
\Ass_R\,M^\vee =\Supp_R M\cap\Ass_R \K_R=\bigcup_{i=1}^n\{PR\mid P\in\mathrm{V}(I_i), \height_{S}P=\height_{S}I_i, M_P\neq (0)\},
$$
so that $M^\vee$ is a torsion-free $R$-module since $\Ass_R\,M^\vee\subseteq \Ass R$.
Moreover $\dim M_\p=\dim R_\p$ for every $\p\in\Supp_RM^\vee$.

\begin{lem}\label{identify}
Suppose that $\Ass R=\Ass_R\K_R$. Let $M$ be a finitely generated torsion-free $R$-module. Then $M^{\vee\vee}$ is naturally isomorphic to an $R$-submodule of $\mathrm{Q}(R)\otimes_RM$ containing $M$.
\end{lem}

\begin{proof}
Let $\psi:M\to M^{\vee\vee}$ denote the natural homomorphism of $R$-modules. Since $\Ass R=\Ass_R\K_R$, we have $\dim R_\q=0$ and $(\K_R)_\q\cong\K_{(R_\q)}$ for all $\q\in\Ass R$, so that $\mathrm{Q}(R)\otimes_R\psi$ is an isomorphism of $\mathrm{Q}(R)$-modules. Thus we have a commutative diagram
$$
\xymatrix@M=10pt{
M \ar@{_{(}->}@<0.3ex>[d]^i \ar[r]^-{\psi } & M^{\vee\vee}\ar@{_{(}->}@<0.3ex>[d]^h  \\
\mathrm{Q}(R)\otimes_RM  \ar[r]^-{\sim }_-{\mathrm{Q}(R)\otimes_R\psi}  &\mathrm{Q}(R)\otimes_RM^{\vee\vee}\ar@{}[lu]|{\circlearrowright}
}
$$
of $R$-modules, where the vertical monomorphisms are canonical (recall that both $M$ and $M^{\vee\vee}$ are torsion-free $R$-modules). We set $\rho=(\mathrm{Q}(R)\otimes_R\psi)^{-1}\circ h.$ Then $M^{\vee\vee}\cong \im \rho$ as an $R$-module and $\im \rho$ is an $R$-submodule of $\mathrm{Q}(R)\otimes_RM$. Hence $M^{\vee\vee}$ is naturally isomorphic to an $R$-submodule of $\mathrm{Q}(R)\otimes_RM$ which contains $M$. 
\end{proof}

We are now in a position to state the second main theorem.

\begin{thm}\label{finiteness1}
Suppose that $\Ass R=\Ass_R\K_R$. Let $M$ be a finitely generated torsion-free $R$-module. Then $\widetilde{M}$ is a finitely generated $R$-module. In addition, the equality $\widetilde{M}=M^{\vee\vee}$ holds in $\mathrm{Q}(R)\otimes_RM$.
\end{thm}

\begin{proof}
Since $\Ass R=\Ass_R\K_R\subseteq \Min_RR$, we have $\Ass R=\Min R$, i.e., $R$ satisfies $(S_1)$. Thus $\Ass_RM\subseteq \Min R$.
By Lemma \ref{identify}, we consider $M^{\vee\vee}$ as an $R$-submodule of $\mathrm{Q}(R)\otimes_RM$ which contains $M$. Set $I=M:_RM^{\vee\vee}$. Then $\height_RI\ge 2$. 
In fact, take $\p\in\mathrm{V}(I)$ and assume $\height_R\p\le 1$. Since  $R_\p$ is a Cohen-Macaulay ring and $M_\p$ is a maximal Cohen-Macaulay $R_\p$-module, we see that $M_\p=(M^{\vee\vee})_\p$. This is a contradiction. Hence $\height_RI\ge 2$. 
As $I(M^{\vee\vee})\subseteq M$, we get $M^{\vee\vee}\subseteq \widetilde{M}$. So $\widetilde{M}=M^{\vee\vee}$ by Proposition \ref{smallest}.
Therefore $\widetilde{M}$ is a finitely generated $R$-module. 
\end{proof}

By combining Theorems \ref{conclusionab} and \ref{finiteness1}, we have the following. 

\begin{cor}
Suppose that $\Ass R=\Ass_R\K_R$. Then there exist $a,b\in\rmW(R)$ such that
$$
\widetilde{R}=\Hom_R(\K_R,\K_R)=\dfrac{R}{a}\cap\dfrac{R}{b}.
$$
\end{cor}
\noindent
Hence, by Corollary \ref{S2ring},  $\widetilde{R}=\Hom_R(\K_R,\K_R)$ is the $(S_2)$-ification of $R$.

\medskip

We close this section by proving the following.

\begin{cor}\label{finiteness2}
Let $R$ be a Noetherian ring and assume that $R$ is a homomorphic image of a Gorenstein ring. Suppose that $R$ is locally quasi-unmixed and satisfies $(S_1)$. Let $M$ be a finitely generated torsion-free $R$-module.  Then $\widetilde{M}$ is a finitely generated $R$-module.
\end{cor}

\begin{proof}
By Theorem \ref{q-unmixed}, we have $\Ass_R\K_R=\Min R = \Ass R$. Hence, Theorem \ref{finiteness1} guarantees that $\widetilde{M}$ is finitely generated as an $R$-module.
\end{proof}

\section{Weakly Arf $(S_2)$-ifications}

In this section, we consider the question of how the weakly Arf property is inherited under the $(S_2)$-ifications. 
Recall that $\rmW(R)$ denotes the set of non-zerodivisors on $R$. 

We begin with the following.

\begin{lem}\label{defArf}
The following three conditions are equivalent.
\begin{enumerate}
\item[$(1)$] The ring $R$ satisfies WAP.\vspace{1mm}
\item[$(2)$] $\overline{aR\, }^{\, 2}=a\,\overline{(aR)}$ for all $a\in\rmW(R)$.
\item[$(3)$] $\dfrac{R}{a}\cap \overline{R}$ is a subring of $\overline{R}$ for all $a\in\rmW(R)$.\vspace{1mm}
\end{enumerate}
\end{lem}

\begin{proof}
For each $a\in\rmW(R)$, we have
$$
\dfrac{R}{a}\cap \overline{R}=\dfrac{a\overline{R} \cap R}{a}
=\dfrac{\overline{aR\, }}{a}
$$
which shows $(2) \Leftrightarrow (3)$. 
On the other hand, let $a \in \rmW(R)$. Then 
\begin{center}
$\dfrac{bc}{a}\in R$ \quad if and only if \quad $\dfrac{b}{a}{\cdot} \dfrac{c}{a}\in \dfrac{R}{a}$, 
\end{center}
where $b, c\in R$. \vspace{1mm}
Hence we get $(1) \Leftrightarrow (3)$.
\end{proof}

In what follows, we assume $R$ is a Noetherian ring.

\begin{thm}\label{mainthm}
Suppose that $\widetilde{R}$ is a finitely generated $R$-module. 
If $R$ satisfies WAP, then so does $\widetilde{R}$.
\end{thm}

\begin{proof}
Assume that $R$ satisfies WAP. Then, for each $a\in\rmW(R)$, the module-finite birational extension $\frac{\overline{aR\,}}{a} = \frac{R}{a}\cap \overline{R}$ of $R$ satisfies WAP (\cite[Proposition 2.10]{CCCEGIM}).  
By Theorem \ref{conclusionab}, there exist $a,b\in\rmW(R)$ such that 
$\widetilde{R}=\frac{R}{a}\cap\frac{R}{b}$.
We then have 
$$
\widetilde{R}=\dfrac{R}{a}\cap\dfrac{R}{b}=\dfrac{R}{a}\cap\dfrac{R}{b}\cap\overline{R}=\dfrac{\overline{aR\,}}{a}\cap\dfrac{\overline{bR\,}}{b}
$$
where the second equality comes from $\widetilde{R}\subseteq \overline{R}$. Since WAP is maintained by taking the intersection (\cite[Lemma 4.5]{EGI}), $\widetilde{R}$ satisfies WAP, as desired. 
\end{proof}

As a consequence of Theorem \ref{mainthm}, we readily get the following.

\begin{cor}\label{integralcor}
Suppose that $R$ is locally quasi-unmixed  and  that $\overline{R}$ is a finitely generated $R$-module. If $R$ satisfies WAP, then so does $\widetilde{R}$.
\end{cor}

\begin{proof}
As $R$ is locally quasi-unmixed, we have $\widetilde{R}\subseteq \overline{R}$. The assertion follows from Theorem \ref{mainthm}, because $\widetilde{R}$ is a module-finite extension of $R$. 
\end{proof}

Even if the normalization $\overline{R}$ is not finitely generated as an $R$-module, the ring $\widetilde{R}$ could be finitely generated; see Corollary \ref{finiteness2}. Hence we have the following.

\begin{cor}\label{imagecor}
Suppose that $R$ is locally quasi-unmixed and satisfies $(S_1)$. Assume that $R$ is a homomorphic image of a Gorenstein ring. If $R$ satisfies WAP, then so does $\widetilde{R}$.
\end{cor}

By Lemma \ref{defArf}, WAP is equivalent to saying that, for each $a \in \rmW(R)$, the reduction number of $\overline{aR\, }$ is at most $1$.
Remember that $R$ is integrally closed if and only if $\overline{aR}=aR$ for every $a \in \rmW(R)$. This indicates that WAP is very close to the integral closedness, which is the reason why the weakly Arf $(S_2)$-ification is of interest next to the integral closure $\overline{R}$, and may have its own significance. 

\begin{defn}
The weakly Arf $(S_2)$-ification of $R$ is the smallest module-finite birational extension of $R$ satisfying both WAP and the condition $(S_2)$ of Serre. 
\end{defn}

The following provides sufficient conditions for the existence of weakly Arf $(S_2)$-ifications.  

\begin{cor}\label{integralcor2}
Suppose that $R$ is a locally quasi-unmixed ring satisfying WAP and $(S_1)$. Then $\widetilde{R}$ is the weakly Arf $(S_2)$-ification,
 if one of the following conditions holds.
\begin{enumerate}
\item[$(1)$] $\overline{R}$ is a finitely generated $R$-module. 
\item[$(2)$] $R$ is a homomorphic image of a Gorenstein ring.
\end{enumerate}
\end{cor}

\begin{proof}
Use Corollaries \ref{S2ring}, \ref{integralcor}, and \ref{imagecor}.
\end{proof}

Finally, we reach the goal of this paper.

\begin{cor}\label{integralcor3}
Let $R$ be a locally quasi-unmixed Noetherian ring which satisfies $(S_1)$. If $\overline{R}$ is a finitely generated $R$-module, then $R$ admits the weakly Arf $(S_2)$-ification.
\end{cor}

\begin{proof}
By \cite[Proposition 7.5]{CCCEGIM} we can construct the smallest module-finite birational extension, denoted by $R^a$, of $R$ which satisfies WAP. Then $R^a$ is a locally quasi-unmixed Noetherian ring satisfying $(S_1)$. Corollary \ref{integralcor} guarantees that $\widetilde{R^a}$ is a module-finite birational extension of $R$ satisfying WAP and the condition $(S_2)$. 
Let $S$ be a module-finite birational extension of $R$ which satisfies WAP and $(S_2)$. Then, by the minimality of $R^a$, we see that $R^a\subseteq S$. Hence $\widetilde{R^a} \subseteq \widetilde{S}=S$. This completes the proof.
\end{proof}

The notion of strict closedness of rings was introduced by Lipman  \cite{L} in connection with the Arf property. We define
$$
R^{*}=\left\{ \alpha \in \overline{R} \mid \alpha \otimes 1 = 1 \otimes \alpha \text{ in } \overline{R} \otimes_R\overline{R}\right\}.
$$
Then $R^*$ is a birational extension of $R$, which is called {\it the strict closure} of $R$ in $\overline{R}$. We say that $R$ is {\it strictly closed}, if $R=R^{*}$ holds. As $(R^*)^*=R^*$ in $\overline{R}$, $R^*$ is strictly closed in $\overline{R}$ (\cite[Section 4, p.672]{L}), so that $[-]^*$ is a closure operation. The strict closure behaves well with respect to the standard operations in ring theory, such as localizations, polynomial extensions, and faithfully flat extensions (see \cite[Proposition 4.2, Lemma 4.9]{CCCEGIM}, \cite[Proposition 4.3]{L}). It is also proved in \cite[Corollary 13.6]{CCCEGIM} that the invariant subrings of strictly closed rings under a finite group action (except the modular case) are also strictly closed. The reader may consult with \cite{L, CCCEGIM, EGI} about further properties.

As Lipman predicted in \cite{L} and as the authors of \cite{CCCEGIM, EGI} are  trying to develop a general theory, there might have a strong connection between the strict closedness and WAP. Indeed, an arbitrary commutative ring $R$ is a weakly Arf ring, once it is strictly closed (\cite[Proof of Proposition 4.5]{L}), and as for Noetherian rings $R$ with $({\rm S}_2)$, it is known by \cite[Corollary 4.6]{CCCEGIM} that $R$ is strictly closed if and only if $R$ satisfies WAP and $R_\fkp$ is Arf for every $\fkp \in \Spec R$ with $\height_R \fkp = 1$. Therefore, provided either $R$ contains an infinite field, or $\height_RM \ge 2$ for every $M \in \Max R$, $R$ is strictly closed if and only if $R$ satisfies WAP (\cite[Corollary 4.6]{CCCEGIM}). 

In this direction, we have the following.

\begin{cor}
Let $R$ be a locally quasi-unmixed Noetherian ring which satisfies $(S_1)$. Suppose that $R$ satisfies WAP and one of the following conditions.
\begin{enumerate}
\item[$(1)$] $\overline{R}$ is a finitely generated $R$-module. 
\item[$(2)$] $R$ is a homomorphic image of a Gorenstein ring.
\end{enumerate}
Then $\widetilde{R}$ is strictly closed if one of the following conditions holds.
\begin{enumerate}
\item[$(\mathrm{i})$] $R$ contains an infinite field. 
\item[$(\mathrm{ii})$] $\height_R M\ge 2$ for every $M \in \Max R$. 
\end{enumerate}
\end{cor}

Closing this paper, we prove the following.

\begin{prop}\label{4.9}
Let $(A,\m)$ be a Noetherian local ring with $\dim A=2$ possessing the canonical module $\rmK_A$. Suppose that $\overline{A}$ is a finitely generated $A$-module and that all the minimal prime ideals of $A$ have the same codimension. We choose non-zerodivisors $a,b\in\m$ on $A$ such that $a\in A:\overline{A}$, $a,b$ is a system of parameters of $A$, and $aA:_Ab=aA:_Ab^2$. 
Then 
$\widetilde{A}=\dfrac{A}{a}\cap\dfrac{A}{b}$
holds.
\end{prop}

\begin{proof}
Since $\overline{A}$ is finitely generated as an $A$-module, we can choose $a\in W(A)\cap (A:\overline{A})\cap\m$. Then 
$
\widetilde{A}=\frac{\rmU(aA)}{a}
$
where $\rmU(aA)$ denotes the minimal component of a primary decomposition of $aA$.
Let $b\in W(A)\cap\m$ such that $a,b$ is a system of parameters of $A$. Since $A$ is Noetherian, we may assume $aA:_Ab=aA:_Ab^2$ by taking a large enough powers of $b$. Hence $aA:_Ab=\bigcup_{n>0}aA:_Ab^n$.
 Since $\dim A=2$, there exists an $\m$-primary ideal $I$ of $A$ such that $aA=\rmU(aA)\cap I$. Thus, for all integers $i\gg 0$, we have $b^i\not\in \rmU(aA)$ and $b^i\in I$. Therefore we obtain $\rmU(aA)=\bigcup_{n>0}aA:_Ab^n$, and hence $\rmU(aA)=aA:_Ab$. 
\end{proof}

\begin{rem}
Proposition \ref{4.9} also holds for graded setting.
More precisely, let $A$ be a Noetherian positively graded ring with $\dim A=2$ possessing the graded canonical module $\rmK_A$. Suppose that $A_0$ is a local ring, $\overline{A}$ is a finitely generated $A$-module, and all the associated prime ideals of $A$ have the same codimension. We choose homogeneous non-zerodivisors $a,b$ on $A$ such that $a\in A:\overline{A}$, $a,b$ is a homogeneous system of parameters of $A$, and $aA:_Ab=aA:_Ab^2$. Then
$\widetilde{A}=\dfrac{A}{a}\cap\dfrac{A}{b}$
holds.
\end{rem}

\begin{ex}[see {\cite[Example 2.6]{EG}}]\label{ex}
Let $T = k[X,Y]$ be the polynomial ring over a field $k$. We consider the subring $R = k[X^5,XY^4, Y^5]$ of $T$. Note that $R$ is a Cohen-Macaulay ring and 
$$
\overline{R} = k[X^5,X^4Y,X^3Y^2,X^2Y^3,XY^4, Y^5]
$$
is a finitely generated $R$-module.
Let $R^a$ denote the weak Arf closure of $R$ (see \cite{CCCEGIM}). By \cite[Example 2.6]{EG}, we then have 
$$
R^a= R[X^9Y^6,X^8Y^7,X^4Y^{11}]=k[X^5,XY^4, Y^5, X^9Y^6, X^8Y^7, X^4Y^{11}]
$$
so that $R^a$ is not a Cohen-Macaulay ring. Set $A=R^a$. 
Since $$\overline{A}=AX^4Y+AX^3Y^2+AX^2Y^3,$$ we get $Y^{10}\in A:\overline{A}$. By setting $a=Y^{10}$ and $b=X^5$, the pair $a,b$ is a system of parameters of $A$.
In addition, we have
$$
aA:_{A}b=aA:_{A}b^2=(Y^{10},X^3Y^{17}, X^4Y^{16})A,
$$
whence
\begin{eqnarray*}
\widetilde{A}&=&\dfrac{A}{a}\cap\frac{A}{b}=\dfrac{aA:_{A}b}{a}=\dfrac{(Y^{10},X^3Y^{17}, X^4Y^{16})A}{Y^{10}} \\
&=&A+AX^3Y^{7}+AX^4Y^{6}=A[X^3Y^7, X^4Y^6].
\end{eqnarray*}
Therefore, the $(S_2)$-ification $\widetilde{R^a}$ is given by 
$$
\widetilde{R^a}=R^a[X^3Y^7, X^4Y^6]~ =R[X^3Y^7, X^4Y^6]~ =k[X^5, XY^4, Y^5, X^3Y^7, X^4Y^6].
$$
This is the weakly Arf $(S_2)$-ification of $R$.
Since $\widetilde{R^a} \ne \overline{R}$, the ring $\widetilde{R^a}$ is not normal. 
\end{ex}

\end{document}